\documentclass[a4paper]{scrartcl}

\usepackage{graphicx}

\usepackage{amsmath}
\usepackage{amssymb}
\usepackage[hypertexnames=false]{hyperref}
\usepackage{cleveref}
\usepackage{autonum}
\usepackage[vmargin=3cm]{geometry}
\usepackage[svgnames,dvipsnames]{xcolor}
\usepackage{enumitem}
\usepackage{amsthm}
\usepackage{mathtools}
\usepackage{mleftright}
\usepackage{bm}
\usepackage{xparse}

\hypersetup{
    colorlinks=true,
    linkcolor=RoyalBlue,
    citecolor=Green
}

\colorlet{refkey}{pink!90!red}
\colorlet{labelkey}{JungleGreen!80!yellow}

\newtheorem{theorem}{Theorem}[section]
\crefname{theorem}{Theorem}{Theorems}

\crefname{proposition}{Proposition}{Propositions}

\crefname{lemma}{Lemma}{Lemmas}

\crefname{corollary}{Corollary}{Corollaries}

\theoremstyle{definition}
\newtheorem{remark}{Remark}[section]
\crefname{remark}{Remark}{Remarks}
\newtheorem{definition}{Definition}
\crefname{definition}{Definition}{Definitions}
\newtheorem{scheme}{Scheme}
\crefname{scheme}{Scheme}{Schemes}

\crefname{section}{Section}{Sections}

\numberwithin{equation}{section}

\numberwithin{equation}{section}
\allowdisplaybreaks

\hypersetup{
    colorlinks=true,
    linkcolor=RoyalBlue,
    citecolor=Green
}

\setlist[enumerate,1]{label=(\arabic*)}

\DeclareMathOperator{\cn}{cn}

\DeclarePairedDelimiterXPP{\inner}[2]{}{\langle}{\rangle}{}{#1,#2}
\DeclarePairedDelimiter{\norm}{\lVert}{\rVert}

\DeclarePairedDelimiter{\paren}{\lparen}{\rparen}
\DeclarePairedDelimiter{\bracket}{\lbrack}{\rbrack}
\DeclarePairedDelimiter{\jump}{\lbrack}{\rbrack}

\providecommand\given{}
\newcommand\SetSymbol[1][]{%
    \nonscript\:#1\vert
    \allowbreak
    \nonscript\:
    \mathopen{}
}
\DeclarePairedDelimiterX{\Set}[1]{\{}{\}}{%
    \renewcommand\given{\SetSymbol[\delimsize]}
    #1
}

\newcommand{\bR}{\mathbb{R}}
\newcommand{\bN}{\mathbb{N}}

\newcommand{\rd}{\mathrm{d}}
\newcommand{\cP}{\mathcal{P}}

\newcommand{\cJ}{\mathcal{J}}
\newcommand{\cJt}{\cJ_\tau}

\title{Higher order discrete gradient method by the discontinuous Galerkin time-stepping method}

\author{    
    Tomoya Kemmochi%
    \thanks{%
        Corresponding author. 
        Graduate School of Engineering, Nagoya University. 
        \\
        Email: \texttt{kemmochi@na.nuap.nagoya-u.ac.jp}, Web: \url{https://t-kemmochi.github.io/en/}
    }
}

\hypersetup{
    pdftitle={Higher order discrete gradient method by the discontinuous Galerkin time-stepping method},
    pdfauthor={T. Kemmochi}
}

\begin{document}
\mleftright

\maketitle

\begin{abstract}
    Many differential equations with physical backgrounds are described as gradient systems, which are evolution equations driven by the gradient of some functionals, and such problems have energy conservation or dissipation properties.
    For numerical computation of gradient systems, numerical schemes that inherit the energy structure of the equation play important roles, which are called structure-preserving.
    The discrete gradient method is one of the most classical framework of structure-preserving methods, which is at most second order accurate.
    In this paper, we develop a higher-order structure-preserving numerical method for gradient systems, which includes the discrete gradient method.
    We reformulate the gradient system as a coupled system and then apply the discontinuous Galerkin time-stepping method.
    Numerical examples suggests that the order of accuracy of our scheme is $(k+1)$ in general and $(2k+1)$ at nodal times, where $k$ is the degree of polynomials.
\end{abstract}

\section{Introduction}

Let $H$ be a real or complex Hilbert space with the inner product $\inner{\cdot}{\cdot}$ and $E$ be a real-valued energy functional defined on an open subset $U \subset H$.
Then, we consider an evolution equation
\begin{equation}
    \begin{cases}
        \dot{u}(t) = L(u(t)) \nabla E(u(t)), & 0 < t \le T, \\
        u(0) = u_0,        
    \end{cases}
    \label{eq:gs}
\end{equation}
where $T>0$ and $u_0 \in H$ are given, $u \colon [0,T] \to H$ is an unknown function, $L$ is a linear operator on $H$ that may be unbounded and may depend on $u$, and $\nabla E \colon U \to H$ is the Fr\'echet derivative of $E$.
In this paper, we call the equation \eqref{eq:gs} the gradient system as it includes gradient flows.
For the solution of the gradient system \eqref{eq:gs}, one obtains
\begin{equation}
    \frac{d}{dt} E[u(t)] = \inner{\nabla E(u)}{\dot{u}} = \inner{\nabla E(u)}{L(u) \nabla E(u)}
    \label{eq:energy}
\end{equation}
by the chain rule.
This property is called the energy structure since it implies the energy conservation if $L$ is skew-symmetric and the energy dissipation if $L$ is negative semi-definite.
Many differential equations with physical backgrounds are described as the gradient system \eqref{eq:gs}.
Typical examples are the Allen--Cahn, Cahn--Hilliard, KdV, and nonlinear Schr\"odinger equations.
For these equations, the energy structure \eqref{eq:energy} reflects certain physical properties.

For numerical computation of gradient systems, structure-preserving numerical methods, which are numerical methods that inherit energy structure of the original equations, play important roles.
It is known that structure-preserving methods not only provide physically reasonable numerical solutions, but also are robust for long-time computations.
Therefore a lot of structure-preserving schemes have been developed in recent decades, such as the discrete gradient method \cite{MR1411343}, the discrete variational derivative method \cite{MR1727636,MR2744841}, the discrete partial derivative method \cite{MR2437123}, the scalar auxiliary variable approach \cite{MR3723659,MR3989239,MR4462611}, and so on (see also \cite{MR2221614}).
Most of these methods are at most second order accurate.

The discrete gradient method is a classical framework to construct structure-preserving methods.
The concept of the discrete gradient plays a crucial role, which is a map $\nabla_\rd E \colon U \times U \to H$ defined through the discrete counterpart of the chain rule 
\begin{equation}
    E[u] - E[v] = \inner{\nabla_\rd E(u,v)}{u-v}
    \label{eq:discrete-gradient-intro}
\end{equation}
for $u,v \in U$ (precise definition will be given in~\cref{def:dg}).
The scheme of the discrete gradient method is then given by 
\begin{equation}
    \frac{u^n - u^{n-1}}{\tau} = L(\bar{u}^n) \nabla_\rd E(u^n,u^{n-1}),
\end{equation}
where $\tau>0$ is the time increment, $u^n \approx u(n \tau)$, and $\bar{u}^n$ is arbitrary approximation of $u^n$ or $u^{n-1}$.
By the definition of the discrete gradient \eqref{eq:discrete-gradient-intro}, it is easy to see that
\begin{equation}
    E[u^n] - E[u^{n-1}] = \tau \inner*{\nabla_\rd E(u^n,u^{n-1})}{L(\bar{u}^n) \nabla_\rd E(u^n,u^{n-1})}
    \label{eq:energy-dgm}
\end{equation}
holds, which corresponds to the energy structure \eqref{eq:energy}.
Notice that the construction of the discrete gradient is not unique.
The averaged vector field method \cite{MR1694701} is one of the unified approaches for construction.

The aim of this paper is to develop a higher-order structure-preserving numerical method for \eqref{eq:gs}.
We focus on the discontinuous Galerkin (DG) time-stepping method \cite{MR826227,MR2249024} to achieve this purpose.
The DG time-stepping method is a temporal discretization method based on the discontinuous Galerkin method in the temporal direction. 
For parabolic problems, many computational and theoretical aspects are studied (see~\cite{MR826227,MR1620144,MR2249024} and references therein).
In particular, it is known that the accuracy at $t=t_n$ is $O(\tau^{2k+1})$ and the accuracy in $J_n$ is $O(\tau^{k+1})$, where $k$ is the degree of polynomials (see also \cite{MR2983919}).

The main idea of our scheme is to reformulate the gradient system \eqref{eq:gs} as a coupled system
\begin{equation}
    \begin{cases}
        \dot{u} = L(u) p, & \text{in } (0,T], \\
        p = \nabla E(u), & \text{in } (0,T], \\
        u(0) = u_0.        
    \end{cases}
    \label{eq:gs-coupled}
\end{equation}
We then apply the standard DG times-stepping method of degree $k$ to the first equation, and a non-standard DG method with test functions of degree $k-1$ to the second equation.
Consequently, there remains one degree of freedom (for the temporal variable), and thus we add an extra equation by using the discrete gradient (see~\cref{scheme:main}).
Since our scheme is based on the Galerkin method, it is easy to apply the scheme to partial differential equations.

We will show that our scheme inherits the energy structure \eqref{eq:energy} (\cref{thm:energy-proposed}).
Moreover, if we use the piecewise constant functions, our scheme coincides with the classical discrete gradient method.
Therefore, the proposed scheme is a higher-order extension of the discrete gradient method.

There are several studies on higher-order structure-preserving methods.
The higher-order discrete gradient method is presented and discussed in \cite{MR3094567,MR4512604}.
In these studies, higher-order discrete gradients are explicitly constructed.
The averaged vector field collocation method \cite{MR2833602,MR2784654,MR3194795} is also a powerful arbitrary higher-order method.
It is pointed out in \cite{MR2983919} that this method is interpreted as a temporal finite element method with continuous trial functions and discontinuous test functions.
Hence it resembles to our scheme. This point will be discussed by numerical experiments.
The continuous stage Runge--Kutta method for conservative systems is discussed in \cite{MR3259826,MR3516866}, and the authors present conditions the conservation laws.
Recently, the DG time-stepping method for dissipative systems is investigated in \cite{MR3987168}.
In this study, the energy dissipative scheme is developed when the energy $E$ is convex.
The current paper can be viewed as a modification of \cite{MR3987168}.

The remainder of the present paper is organized as follows.
In Section~\ref{sec:dgm}, we review the classical discrete gradient method.
Then, we propose the novel scheme in Section~\ref{sec:proposed}.
Numerical examples are given in Section~\ref{sec:numerical} to see the performance of the scheme.
Finally we conclude this paper with some remarks in Section~\ref{sec:conclusion}.

\section{Discrete gradient method} 
\label{sec:dgm}

In this section, we review the classical discrete gradient method according to \cite{MR1411343}.
We first introduce the discrete gradient.
\begin{definition}\label{def:dg}
    Let $E \colon U \to \bR$ is a functional defined on an open subset $U \subset H$.
    Then, a map $\nabla_\rd E \colon U \times U \to H$ is called the \emph{discrete gradient} if the following two properties are satisfied:
    \begin{enumerate}[label=(\roman*)]
        \item $E[u] - E[v] = \inner{\nabla_\rd E(u,v)}{u-v}$ for all $u,v \in U$,
        \item $\nabla_\rd E(u,u) = \nabla E(u)$ for all $u \in H$.
    \end{enumerate}
\end{definition}

Once the discrete gradient is constructed, the scheme of the discrete gradient method for the gradient system \eqref{eq:gs} is given by, for example, 
\begin{equation}
    \frac{u^n - u^{n-1}}{\tau} = L\paren*{\frac{u^n + u^{n-1}}{2}} \nabla_\rd E(u^n,u^{n-1}),
\end{equation}
where $u^n \approx u(n\tau)$, and then the discrete energy structure \eqref{eq:energy-dgm} holds.
This scheme is at most second order.

The construction of the discrete gradient is not unique.
We here give three examples, which lead to at most second order schemes.

\subsection*{Discrete gradients of Gonzalez}
Gonzalez \cite{MR1411343} proposed the discrete gradient
\begin{equation}
    \nabla_\rd E(u,v) \coloneqq \nabla E \paren*{\frac{u+v}{2}} + \frac{E[v] - E[u] - \inner*{\nabla E \paren*{\frac{u+v}{2}}}{v-u}}{\norm{v-u}^2}(v-u),
    \label{eq:Gonzalez}
\end{equation}
for $u,v \in H, \, u \ne v$.

\subsection*{Averaged vector field}
The averaged vector field \cite{MR1694701} is given by
\begin{equation}
    \nabla_\rd E(u,v) \coloneqq \int_0^1 \nabla E( (1-s)u + sv ) ds.
    \label{eq:AVF}
\end{equation}

\subsection*{Discrete gradients of Itoh--Abe}
When $H = \bR^N$ and $E[u] = E(u_1,\dots,u_N)$, the discrete gradient of Itoh--Abe \cite{MR0943488} is given by
\begin{equation}
    \nabla_\rd E(u,v)_j \coloneqq \frac{E(v_1,\dots,v_{j-1},u_j,u_{j+1},\dots,u_N) - E(v_1,\dots,v_{j-1},v_j,u_{j+1},\dots,u_N)}{u_j - v_j}
    \label{eq:IA}
\end{equation}
for $j=1,\dots,N$.

\section{Proposed scheme}
\label{sec:proposed}

In this section, we present our novel higher-order structure-preserving scheme for the gradient system \eqref{eq:gs} based on the DG time-stepping method.
We first introduce the space of piecewise polynomials.
Let $0 = t_0 < t_1 < \dots < t_N = T$ be a temporal mesh and $J_n = (t_{n-1},t_n]$ be a subinterval for $n=1,\dots,N$, and set $\tau_n = |J_n|$, $\tau = \max_n \tau_n$, and $\cJt = \{ J_n \}_{n=1}^N$.
Then, we define the space of $H$-valued piecewise polynomials of degree $k$ by 
\begin{align}
    \cP^k(J_n;H) &\coloneqq \Set*{ \sum_{j=0}^k v_j t^j \given v_j \in H, \forall j, \, t \in J_n}, \\ 
    V_\tau = V_\tau(H) &\coloneqq \Set*{ v \in L^2(0,T; H) \given v|_{J_n} \in \cP^k(J_n;H), \forall n }.
\end{align}
For $v \in V_\tau$, we define the nodal values and the jump by 
\begin{equation}
    v^n \coloneqq v(t_n), \quad v^{n,+} \coloneqq \lim_{t \downarrow t_n} v(t), \quad \jump{v}^n \coloneqq v^{n,+} - v^n.
\end{equation}

We now propose our novel scheme.
We reformulate the gradient system \eqref{eq:gs} as
\begin{equation}
    \begin{cases}
        \dot{u}(t) = L(u(t)) p(t), & 0 < t \le T, \\
        p(t) = \nabla E(u(t)), & 0 < t \le T, \\
        u(0) = u_0.        
    \end{cases}
    \label{eq:gs-2}
\end{equation}
Then, our scheme is formulated as follows.

\begin{scheme}\label{scheme:main}
    Find $(u_\tau,p_\tau) \in V_\tau \times V_\tau$ that satisfies $u_\tau^0 = u_0$ and
    \begin{equation}
        \begin{dcases}
            \int_{J_n} \inner{\dot{u}_\tau}{v} dt + \inner*{\jump{u_\tau}^{n-1}}{v^{n-1,+}} = \int_{J_n} \inner{L(u_\tau) p_\tau}{v} dt, & \forall v \in \cP^k(J_n;H), \\
            \int_{J_n} \inner{p_\tau}{w} dt = \int_{J_n} \inner{\nabla E(u_\tau)}{w} dt, & \forall w \in \cP^{k-1}(J_n;H), \\
            p^{n-1,+}_\tau = \nabla_\rd E(u_\tau^{n-1,+}, u_\tau^{n-1})
        \end{dcases}
        \label{eq:proposed}
    \end{equation}
    for $n=1,2,\dots,N$. Here, $\nabla_\rd E$ is the discrete gradient in the sense of \cref{def:dg}.
\end{scheme}

\begin{remark}\label{rem:extension}
    When $k=0$, the scheme can be viewed as the classical discrete gradient method.
    Indeed, since $u_\tau|_{J_n} \equiv u^n = u^{n-1,+}$ when $k=0$, we have
    \begin{equation}
        \begin{dcases}
            \inner*{u_\tau^n - u_\tau^{n-1}}{v} = \tau_n \inner{L(u_\tau^n) p_\tau^n}{v}, & \forall v \in H, \\
            p^n_\tau = \nabla_\rd E(u_\tau^n, u_\tau^{n-1}),
        \end{dcases}
    \end{equation}
    which is equivalent to 
    \begin{equation}
        \frac{u_\tau^n - u_\tau^{n-1}}{\tau_n} = L(u_\tau^n) \nabla_\rd E(u_\tau^n, u_\tau^{n-1}).
    \end{equation}
    In particular, when $L$ is independent of $u$, this coincides with the discrete gradient method.
\end{remark}

Considering the nodal values $\{ u_\tau^n \}_{n=0}^N$, we regard the DG scheme \eqref{eq:proposed} as a one-step method,
and this inherits the energy structure of \eqref{eq:gs}.

\begin{theorem}\label{thm:energy-proposed}
    Assume that the numerical scheme \eqref{eq:proposed} has a unique solution and let $(u_\tau,p_\tau) \in V_\tau \times V_\tau$ be the solution.
    Then, we have
    \begin{equation}
        E[u_\tau^n] - E[u_\tau^{n-1}] = \int_{J_n} \inner{L(u_\tau) p_\tau}{p_\tau} dt
        \label{eq:energy-proposed}
    \end{equation}
    for each $n$. 
    In particular, if $L(v)$ is skew-symmetric (resp.\ negative semi-definite) for all $v \in H$, then the energy conservation $E[u_\tau^n] = E[u_\tau^{n-1}]$ (resp.\ the energy dissipation $E[u_\tau^n] \le E[u_\tau^{n-1}]$) holds.
\end{theorem}

\begin{proof}
    Letting $v=p_\tau$ in the first equation of \eqref{eq:proposed} and using the third equation, we have
    \begin{align}
        \int_{J_n} \inner{\dot{u}_\tau}{p_\tau} dt 
        &= \int_{J_n} \inner{L(u_\tau) p_\tau}{p_\tau} dt - \inner*{\jump{u_\tau}^{n-1}}{p_\tau^{n-1,+}} \\
        &= \int_{J_n} \inner{L(u_\tau) p_\tau}{p_\tau} dt - \inner*{\jump{u_\tau}^{n-1}}{\nabla_\rd E(u_\tau^n, u_\tau^{n-1})} \\
        &= \int_{J_n} \inner{L(u_\tau) p_\tau}{p_\tau} dt - \paren*{E[u_\tau^{n-1,+}] - E[u_\tau^{n-1}]}.
    \end{align}
    Moreover, letting $w = \dot{u}_\tau$ in the second equation of \eqref{eq:proposed}, we have
    \begin{equation}
        \int_{J_n} \inner{p_\tau}{\dot{u}_\tau} dt = \int_{J_n} \inner*{\nabla E(u_\tau)}{\dot{u}_\tau} dt = E[u^n_\tau] - E[u_\tau^{n-1,+}].
    \end{equation}
    These two equations implies the desired identity \eqref{eq:energy-proposed}.
\end{proof}

\begin{remark}
    In the above theorem, we assumed that the scheme \eqref{eq:proposed} has a unique solution.
    However, we do not address this issue in the current paper and postpone it as future work.
\end{remark}

\section{Numerical examples}
\label{sec:numerical}

\subsection{Scalar ODE}
We consider the simple scalar ODE 
\begin{equation}
    \frac{du}{dt} = u-u^3 = -f'(u), \quad f(u) = \frac{(1-u^2)^2}{4}
    \label{eq:scalar-ode}
\end{equation}
with the corresponding energy $E[u] = f(u)$ to see the energy dissipation and convergence rate.
In this case, the three of the discrete gradients \eqref{eq:Gonzalez}, \eqref{eq:AVF}, and \eqref{eq:IA} coincide and one has
\begin{equation}
    \nabla_\rd E(u,v) = \frac{u^3 + u^2v + uv^2 + v^3}{4} - \frac{u+v}{2} .
\end{equation}
Moreover, the exact solution is given by
\begin{equation}
    u(t) = \frac{u_0}{\sqrt{(1-e^{-2t})u_0^2 + e^{-2t}}}
\end{equation}
if $u(0) = u_0 > 0$.

\begin{figure}
    \centering
    \includegraphics[page=1,width=0.32\linewidth]{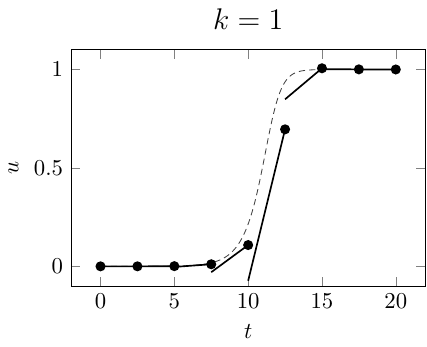}
    \includegraphics[page=2,width=0.32\linewidth]{ode.pdf}
    \includegraphics[page=3,width=0.32\linewidth]{ode.pdf}
    \caption{Numerical examples of the scheme \eqref{eq:proposed} for the ODE \eqref{eq:scalar-ode} when $k=1,2,3$. The black lines are the numerical solutions, the black dots are nodal values, and the gray dashed lines are the exact solution.}
    \label{fig:scalar-ode}
\end{figure}

We set the initial value $u_0 = 10^{-5}$, the uniform time increment $\tau_n \equiv \tau = 20/8$, and the final time $T=20$.
We computed the above ODE by our scheme \eqref{eq:proposed} for $k=1,2,3$.
The results are plotted in Figure~\ref{fig:scalar-ode}.
Although the time increment $\tau$ is very big, our scheme works well.
Moreover, energy dissipation $E[u^n] \le E[u^{n-1}]$ is observed.
Because of the discontinuity, energy dissipation does not hold at every time.
In particular, one observe that $E[u^n] > E[u^{n,+}]$ holds when $k=1$.

\begin{figure}
    \centering
    \includegraphics[page=4,width=0.45\linewidth]{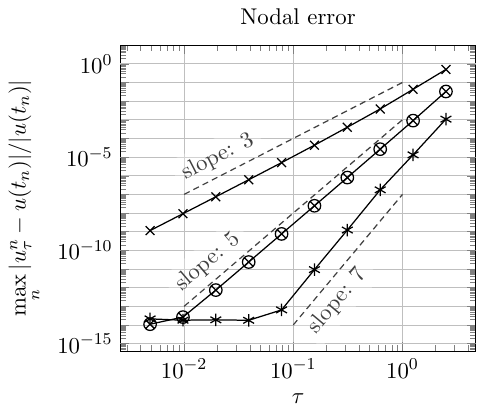}
    \includegraphics[page=5,width=0.45\linewidth]{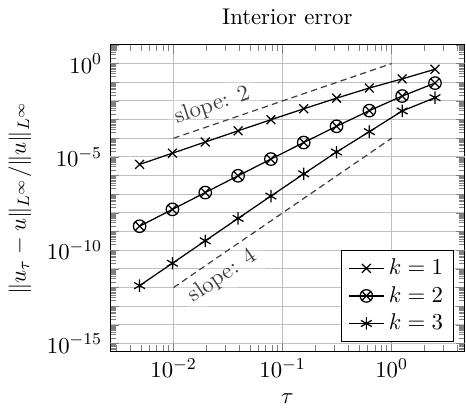}
    \caption{Convergence rate of the scheme \eqref{eq:proposed} for the ODE \eqref{eq:scalar-ode} when $k=1,2,3$.
    The left one shows the nodal errors and the right one shows the interior errors.}
    \label{fig:scalar-ode-rate}
\end{figure}

The relative errors with variable time increments are plotted in Figure~\ref{fig:scalar-ode-rate}.
We plotted the nodal error $\max_n | u^n_\tau - u(t_n)|/|u(t_n)|$ and the interior error $\| u_\tau - u\|_{L^\infty(0,T)}/\|u\|_{L^\infty(0,T)}$. 
The result suggests that the nodal error is $O(\tau^{2k+1})$ and the interior error is $O(\tau^{k+1})$, which coincides with the convergence rate of the usual DG time-stepping method.

\subsection{KdV equation}

Let us consider the PDE case.
We consider the KdV equation
\begin{equation}
    \partial_t u + 6u\partial_x u + \partial_{xxx} u = 0, \qquad 
    \text{in } (0,L) \times (0,T)
    \label{eq:kdv}
\end{equation}
with the periodic boundary condition.
The corresponding energy is 
\begin{equation}
    E[u] = \int_0^L \paren*{ \frac{1}{2}u_x^2 - u^3 } dx 
\end{equation}
and the gradient is $\nabla E(u) = -u_{xx} -3u^2$.
Therefore, the KdV equation is reformulated as
\begin{equation}
    \partial_t u = p, \qquad p = -u_{xx} -3u^2.
\end{equation}
When the initial function is 
\begin{equation}
    u_0(x) = \alpha + 2\kappa^2 m^2 \cn^2 (\kappa x | m),
\end{equation}
the exact solution is the cnoidal wave
\begin{equation}
    u_\text{ex}(x,t) = \alpha + 2\kappa^2 m^2 \cn^2 (\kappa(x-ct) | m), \qquad c = 6\alpha + 4(2m^2-1)\kappa^2
\end{equation}
with spatial period $L \coloneqq 2K(m)/\kappa$ and temporal period $L/|c|$,
where $\cn$ is one of the Jacobi elliptic functions and $K(m)$ is the complete elliptic integral of the first kind.

We use the conforming finite element method with degree $l$ for the spatial discretization.
We state the complete scheme, let $N_x \in \bN$, $h = L/N_x$, $x_j = jh$, and $\Omega_j = (x_{j-1},x_j)$.
Then, define the conforming $P^l$ finite element space with periodicity by
\begin{equation}
    S_h \coloneqq \Set{ v_h \in H^1(0,L) \given v_h|_{\Omega_j} \in \cP^l(\Omega_j; \bR),\, v_h(0) = v_h(L) }.
\end{equation}
We finally set $V_{\tau,h} \coloneqq V_\tau(S_h)$. 
That is, we set $H = S_h$ with the $L^2$-inner product.

Then, our scheme \eqref{eq:proposed} for the KdV equation \eqref{eq:kdv} is given as follows.
Find $(u,p) \in V_{\tau,h} \times V_{\tau,h}$ such that
\begin{equation}
    \begin{dcases}
        \int_{J_n} (\tau_t u, v) dt + (\jump{u}^{n-1} ,v^{n-1,+}) = \int_{J_n} (p_x, v) dt, & \forall v \in \cP^k(J_n; S_h), \\
        \int_{J_n} (p,w) dt = \int_{J_n} \bracket*{ (u_x, w_x) - (3u^2, w)  } dt, & \forall w \in \cP^{k-1}(J_n; S_h), \\
        (p^{n-1,+}, q) = (\nabla_\rd E(u^{n-1}, u^{n-1,+}), q) , & \forall q \in S_h,
    \end{dcases}
    \label{eq:scheme-kdv}
\end{equation}
where $(\cdot,\cdot)$ is the usual $L^2$-inner product over $(0,L)$.
Here, we define the discrete gradient by
\begin{equation}
    (\nabla_\rd E(v,\bar{v}), q) \coloneqq \paren*{ \frac{v_x+\bar{v}_x}{2}, q_x } - \paren*{ \frac{v^2 + v\bar{v} + \bar{v}^2}{3}, q }, \qquad \forall q \in S_h
\end{equation}
for $v,\bar{v} \in S_h$.

We set the parameters as follows:
\begin{equation}
    m = \sqrt{0.9}, \quad \kappa = 1, \quad \alpha = 0.
\end{equation}
Moreover, we set the degree of spatial polynomials by $l=2k$, we used the uniform temporal mesh $t_n = n\tau$ with $\tau = T/N_t$,
and we set $N_t=N_x=2^i$ with $i \in \bN$.
Then we computed the KdV equation for one period and calculated the $L^\infty$-error at nodal times $\max_n \| u^n(\cdot) - u_\text{ex}(t_n) \|_{L^\infty(0,L)}/\| u_\text{ex}(\cdot,t_n) \|_{L^\infty(0,L)}$ for variable $i$.

\begin{figure}
    \centering
    \includegraphics[width=0.5\linewidth]{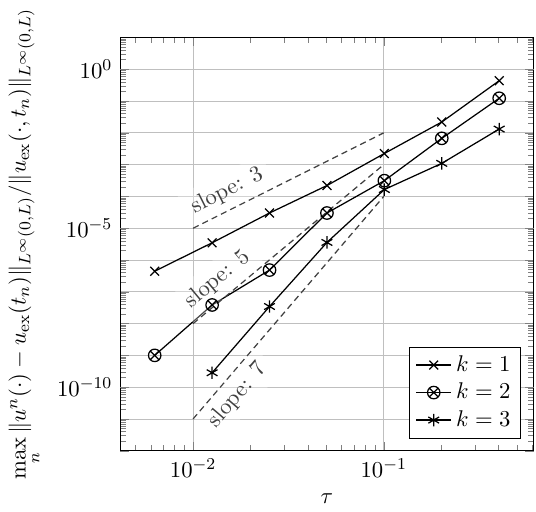}
    \caption{Numerical results of \eqref{eq:scheme-kdv} for $k=1,2,3$.}
    \label{fig:kdv}
\end{figure}

The result is plotted in Figure~\ref{fig:kdv}.
This suggests that the nodal accuracy is still $O(\tau^{2k+1})$ even in the PDE case.

\section{Concluding remarks}
\label{sec:conclusion}

In this paper, we proposed a higher-order structure-preserving method for the gradient system \eqref{eq:gs} by combining the discrete gradient method and the discontinuous Galerkin time-stepping method.
Our scheme \eqref{eq:proposed} inherits the energy structure in the sense of \eqref{eq:energy-proposed}.
Since the scheme is based on the Galerkin method, it is easy to apply our scheme to PDEs.
Numerical results suggest that the temporal accuracy of the scheme is $O(\tau^{2k+1})$ at nodal times and $O(\tau^{k+1})$ at interior times.
However, theoretical aspects of the scheme, such as the well-posedness and error estimates, are not addressed in the present paper.
Moreover, relationship between our scheme and the existing schemes is not trivial.
These issues should be addressed in the future study.

\bibliographystyle{plain}
\bibliography{reference}
\end{document}